\documentclass[12pt]{article}
\usepackage{fullpage}
\usepackage{epsf,epsfig,amsfonts,amsgen,amsmath,amstext,amsbsy,amsopn,amsthm,amssymb,epstopdf,tikz,mathrsfs}
\usepackage{ebezier,eepic,mathtools,dsfont,lmodern}
\usepackage{textcomp}
\usepackage{color,colordvi}
\usepackage{multirow}
\usepackage{url}
\usepackage{graphicx}
\usepackage{epsf}
\usepackage[format=hang, margin=10pt]{caption}
\usepackage{lmodern}
\newtheorem{theorem}{Theorem}[section]
\newtheorem{proposition}[theorem]{Proposition}

\newtheorem{conjecture}[theorem]{Conjecture}
\newtheorem{lemma}[theorem]{Lemma}

\newtheorem{question}[theorem]{Question}

\theoremstyle{definition}

\newtheorem{example}[theorem]{Example}
\newtheorem{remark}[theorem]{Remark}

\newcommand{\normcycles}[1]{\left\lVert #1\right\rVert}
\newcommand{\expect}{\mathbb{E}} 
\newcommand{\var}{\mathbf{Var}} 
\newcommand{\convd}{\xrightarrow{d}} 
\newcommand{\normaldist}{\mathcal{N}} 

\newcommand{\Var}{\operatorname{Var}}
\newcommand{\E}{\mathbb{E}}

\newcommand{\graphseq}{\mathcal{G}} 

\def\dsum{\displaystyle\sum}

\begin{document}

\title{On the Asymptotic Normality and Unimodality of Genus Distributions of Wheels}

\author{Yichao Chen\thanks{School of Mathematics, SuZhou University of Science and Technology, Suzhou 215009, P.R. China. Email: chengraph@163.com. Partially supported by the NNSFC under Grant No. 12271392}\,
\qquad
Yan Yang\thanks{School of Mathematics and KL-AAGDM, Tianjin University, Tianjin 300354, P.R. China. Email: yanyang@tju.edu.cn. Supported by National Natural Science Foundation of China (No. 12371350).} \,}

\date{}

\maketitle

 \begin{abstract}
The genus polynomial of a graph is the generating polynomial for the
number of nonequivalent embeddings of the graph on each orientable surface. In this paper, we address three questions  on
genus polynomials for wheel graphs: the computation of genus polynomials, the unimodality and the asymptotic normality of their coefficients. We derive an explicit formula for the genus polynomial of wheel graphs by
combining methods of the joint tree model and characters theory, and then prove its real-rootedness. This stronger result implies the log-concavity,
unimodality, and asymptotic normality of its coefficients. Thus, we
confirm the unimodality conjecture for the genus distribution of wheel graphs and provide a positive answer to
the asymptotic normality question  posed by Zhang, Peng, and Chen
(\emph{Adv. in Appl. Math.} \textbf{127} (2021), 102175).


\end{abstract}

\section{Introduction}
Graph embeddings on surfaces is a central topic in topological graph theory. A {\it surface} is a compact 2-dimensional manifold without boundary. In this paper, we only consider  the orientable surface, which is obtained from the sphere $\mathbb{S}_0$ by adding handles. If we add $g$ handles to $\mathbb{S}_0$, then we get the orientable surface of genus $g$, denoted by $\mathbb{S}_g$. A surface can be represented by a polygon of even number of edges in the plane, whose edges are pairwise identified and directed clockwise or counterclockwise such that each paired edges have opposite directions. Such polygonal representations of surfaces can be written by words as follows. For the
sphere $\mathbb{S}_0=a_0a_0^{-1}$ and
 $\mathbb{S}{g}=\prod\limits_{i=1}^{g}a_{i}b_{i}a_{i}^{-1}b_{i}^{-1}$
for the orientable surface of genus $g$~$(g\geq 1)$, where $a_i^{-}$, $b_i^-$ are with the opposite direction of $a_i$, $b_i$ on the boundary of the polygon.

An {\it embedding} of a graph $G$ on a surface $S$ is a homeomorphism $h: G\rightarrow S$ of $G$ on $S$ such that every component of
$S-h(G)$ is a 2-cell. Two embeddings $h: G\rightarrow S$ and $g: G\rightarrow S$ of a graph $G$ on a surface $S$ are said to be {\it equivalent} if there is an orientation-preserving homeomorphism $f: S\rightarrow S$ such that $f\circ h=g.$

Let $g_{k}(G)$ be the number of equivalence classes of embeddings of a graph $G$ on the orientable surface of the genus $k$. The sequence $g_0(G),~g_1(G),~g_2(G),\cdots$
is called the {\it genus distribution} of $G$.  And  the polynomial $$\Gamma_{G}(x)=\sum\limits_{k\geq 0}g_{k}(G)x^{k}$$ is called the {\it genus polynomial} of $G$.

The theory of genus distributions is a long-standing topic in the study of graph embeddings, because it incorporates not only  structural properties of a graph, such as its minimum genus and maximum genus, but also enumeration properties of the graph on all surfaces.

Research on genus distributions has progressed along three main directions: \textbf{the computation of genus polynomials}, \textbf{the unimodality of genus distributions}, and \textbf{the asymptotic normality of genus distributions}. These three topics range from computational methods to the algebraic and probabilistic properties of it, and together they have profoundly shaped our understanding of graph embeddings. In this paper, we address all three questions in the context of wheel graphs.

For $n\geq 3$, the {\it wheel graph} of $n$ spokes is the graph $W_{n}$ obtained from the cycle $C_{n}$ by adding a new vertex and joining it to all vertices of $C_{n}$.  Wheel graphs are fundamental objects in graph theory due to their important role in the study of 3-connected graphs~\cite{Tutte1961}. Many polynomial functions associated with wheel graphs have also been studied, such as,
Lu, Xie and Yang~\cite{LXY2022} obtained explicit formulas for Kazhdan-Lusztig polynomials of wheel matroids and further proved that these polynomials have only negative real roots, etc.

Below we briefly introduce the three topics that this paper focuses on.

\subsection{Computation of genus polynomials}
Gross and Furst \cite{GrossFurst1987} inaugurated the problem of genus
distribution of graph embeddings on orientable surfaces when
they were studying the hierarchy for embedding invariants of a
graph. Since then, the genus polynomials of various families of graphs have been computed, for example, bouquets \cite{GRT89};
closed-end ladders and cobblestone paths \cite{FurstGrossStatman1989}; dipoles \cite{KL93}; Ringel ladders \cite{Tesar2000}; star-ladders \cite{Chen2012a}, etc.
In order to derive the genus polynomials of a broader range of graphs, the genus distributions under graph operations have also been studied, see \cite{Gross2011b, Khan2010, Poshni2012}, etc. Nevertheless, the computation of the genus polynomials for graphs remains a highly challenging problem.

\subsection{Unimodality of genus distributions}
A sequence $a_0,a_1,\dots,a_n$ of nonnegative real numbers is \emph{log-concave} if $$a_k^2 \ge a_{k-1}a_{k+1} ~~~\mbox{for}~~~ 1\le k\le n-1,$$ and is \emph{unimodal} if for some $0\le j\le n$ we have $$a_0\le a_1\le\cdots\le a_j\ge a_{j+1}\ge a_{j+2}\ge\cdots\ge a_n.$$
In combinatorics, log-concavity is regarded as a strong structural property, it implies unimodality and can arise from real-rootedness of the associated generating polynomial. Applied to genus distributions, Gross, Robbins, and Tucker~\cite{GRT89} proposed the following conjecture.

\begin{conjecture}[\cite{GRT89}]\label{c1:1}
The genus distribution of every graph is log-concave.
\end{conjecture}

Conjecture \ref{c1:1} has been verified for some classes of graphs, for example, bouquets~\cite{GRT89}, dipoles~\cite{KL93}, ring-like families~\cite{GMT14}, circular ladders~\cite{Chen15}, Ringel ladders \cite{Gross2015a}, some vertex and edge-amalgamations graphs~\cite{GMTW15}, etc.
One of the powerful tool for establishing log-concavity in this context is real-rootedness: if the genus polynomial $\Gamma_G(x)$ has only real roots, then its coefficients sequence is automatically log-concave.
In~\cite{Stahl}, Stahl conjectured the following stronger property of genus polynomials.

\begin{conjecture}[\cite{Stahl}]\label{c1:2}
The genus polynomial of every graph has only real roots
\end{conjecture}

Stahl \cite{Stahl} proved that Conjecture \ref{c1:2} holds for some classes of
graphs including bouquets, dipoles, and cobblestone paths, etc. Chen~\cite{Chen08}  verified  Conjecture \ref{c1:2} for all graphs
with maximum genus 2.  Gross, Mansour, Tucker and Wang~\cite{Gross2016} confirmed Conjecture \ref{c1:2} for iterated claws.

Conjecture \ref{c1:2} remained open until 2010, at which time Chen and Liu
\cite{ChenLiu10} found two cubic graphs, whose genus polynomials both have
non-real roots. Carr, Dhaliwal and Mohar identified more cubic graphs as counterexamples to Conjecture \ref{c1:2} in \cite{CarrDhaliwalMohar2025}. And Conjecture \ref{c1:1}, which had stood for more than three decades, was eventually disproved by Mohar \cite{Mohar2026}  in 2026.

Although neither Conjecture \ref{c1:1} nor Conjecture \ref{c1:2} holds for all graphs, it remains meaningful to investigate for which classes of graphs they do hold. Moreover,
as mentioned in \cite{Mohar2026}, the conjecture regarding the weaker property of unimodality is still open.

\begin{conjecture}\label{c1:3}
The genus distribution of every graph is unimodal.
\end{conjecture}

\subsection{Asymptotic normality of genus distributions}
As the order of a graph increases, it is natural to ask whether its genus
distribution approaches a standard probability distribution. Of particular
interest is whether the genus distribution is asymptotically normal.

Let $(X_n)_{n\geq 1}$ be a sequence of real-valued random variables, and let $X$ be a real-valued random variable with cumulative distribution
function $F_X$. We say that $X_n$ \emph{converges in distribution} to $X$,
denoted by $X_n \convd X,$ if $
\lim_{n\to\infty} F_{X_n}(x)=F_X(x)
$ for every continuity point $x\in R$ of $F_X$, where $F_{X_n}$ denotes the cumulative distribution
function of $X_n$. A sequence of real-valued random variables $(X_n)_{n\geq 1}$ is said to be
\emph{asymptotically normal} if
$$
\frac{X_n-\expect X_n}{\sqrt{\var X_n}}
\convd \normaldist(0,1)
\qquad\text{as } n\to\infty,$$ where $\normaldist(0,1)$ denotes the standard normal distribution. Let $\graphseq=(G_n)_{n\geq 1}$ be a sequence of graphs. We say that the genus distributions of $\graphseq$ are \emph{asymptotically normal} if the associated sequence of genus random variables $(g(G_n))_{n\geq 1}$ is asymptotically normal.

In 2021, Zhang, Peng, and Chen \cite{ZPC2021} proved that the genus distributions of an $H$-linear family of graphs with spiders are asymptotically normal. Zhang, Peng, and Zhang \cite{ZPZ2021} proved that  the Euler-genus distributions
of any ladder-like sequence of graphs are asymptotically normal.
In \cite{ZPC2021}, the authors asked the following question.

\begin{question}[\cite{ZPC2021}]\label{q1:4}
 Are the genus distributions of bouquets, dipoles and wheel graphs asymptotically normal?
\end{question}

For bouquets and dipoles , Chen and Fang \cite{ChenFang2022} provided an affirmative answer to Question \ref{q1:4}.

\subsection{Main results and organizations}
In this paper, an explicit formula for the genus polynomial of wheel graphs is derived, and its real-rootedness is proved. This stronger result implies the log-concavity, unimodality, and asymptotic normality of the genus distribution of wheel graphs. Thus, we confirm Conjecture \ref{c1:3} for wheel graphs and provide a positive answer to Question \ref{q1:4}.

We first clarify the notation. We use $\left[ {n \atop k} \right]$ for the signless Stirling numbers of the first kind, and use \(x^{\overline{n}}\) to denote the rising factorial $x(x+1)\cdots(x+n-1)$.

The main results of this paper are as follows.

\begin{theorem}
\label{thm:distribution}
The genus polynomial of the wheel graph $W_n$ $(n\geq 3)$
\begin{eqnarray*}
\Gamma_{W_n}(x) &=& \frac{4}{n(n+1)} \sum_{m=0}^{\lfloor n/2 \rfloor} \begin{bmatrix} n+1 \\ n-2m \end{bmatrix} x^m
+ \frac{2(-1)^n}{n} x^{(n+1)/2} \sum_{k=0}^{n-1} c_k \left(k - n + 1 - \frac{1}{\sqrt{x}}\right)^{\overline{n}},
\end{eqnarray*}
in which
\[
c_k =
\begin{cases}
2^{n-1} - 1, & k = 0, \\[4pt]
\displaystyle \sum_{j=k+1}^{n-1-k} \binom{n-1}{j}, & 1 \le k \le n-2, \\[4pt]
-\bigl(2^{n-1} - 1\bigr), & k = n-1.
\end{cases}
\]
\end{theorem}

\begin{theorem}
\label{thm:main}
The genus polynomial of the wheel graph $W_n$ $(n\geq 3)$ has only
negative real roots. Consequently, the genus distribution of \(W_n\) is log-concave and unimodal.
\end{theorem}

\begin{theorem}
\label{thm:normal}
The genus distributions of \((W_n)_{n\geq3}\) are  asymptotically normal with mean
$\frac{n}{2}$ and variance $\frac14\ln n.$
\end{theorem}

The rest of this paper is organized as follows.
In Section \ref{Sec:2}, we prove Theorem \ref{thm:distribution}.
In Section \ref{Sec:3}, we prove Theorem \ref{thm:main}, and in Section \ref{Sec:4}, we prove Theorem \ref{thm:normal}.

\section{Genus Polynomials of Wheels}\label{Sec:2}
In this section, we will give a new derivation of the genus distribution of wheel graphs. Within this framework, we first use the joint tree model to establish connections among the genus distribution of wheel graphs and those of dipoles and tripolar graphs. Then, we compute  the genus distributions of tripolar graphs and dipoles by means of the character method, and consequently obtain the genus distribution of wheel graphs.

\subsection{Joint tree model}
When we describe an embedding of a graph on a surface, the most popular embedding model announced by Edmonds in 1960 is based on the rotation system at vertices. See e.g., Sections $3.2$ and $3.3$ in \cite{MT01} for detail. About two decades ago, another embedding model, called the joint tree model, was introduced by Liu \cite{liu08}. The basic idea is to transfer the rotation system into a cyclically ordered string of letters, which is called an associate surface. Many results about the genus, average genus and genus distribution of a graph, such as \cite{Chen10,SL08,SL10a,SL10,WL08,YL07} etc, have been obtained by using the joint tree model.

In the joint model, we do the following steps to get a joint tree of $G$ which is corresponding to an embedding of $G$ on a surface.
\\{\bf Step 1.}  Choose a spanning tree $T$ of a graph $G$ arbitrarily. Then the edge set $E$ of $G$ can be
partitioned into $E_{T}$ (tree edge) and $\overline{E}_{T}$ (cotree
edge).
\\{\bf Step 2.} For $1\leq i\leq \beta$,
we split each cotree edge $e_{i}\in \overline{E}_{T}$ into two semi-edges and label them by the same letter as $a_{i}$ where $\beta$ is the betti
number of $G$. Then we get a tree with $|V(G)|+2\beta$ vertices in which $2\beta$ of them are leaves.
\\{\bf Step 3.} Index the $2\beta$ semi-edges by $+$ (always omitted) or $-$ such that the indices of each
pair of semi-edges labelled with same letter are distinct.
\\{\bf Step 4.} A {\it rotation} at a vertex $v,$ denoted by $\sigma_{v},$  is a
cyclic permutation of semi-edges incident with $v$. Let $\sigma_{G}=\prod_{v\in V(G)}\sigma_{v}$  be a rotation system of
$G$. The tree with an index of each semi-edge and a rotation system of  $G$  is called a {\it joint tree} of $G$.

By reading these lettered
semi-edges with indices of a joint tree in a
fixed orientation (clockwise or counterclockwise), we can get an
algebraic representation for a surface. It is a cyclic order of
$2\beta$ letters with indices. This surface is called an {\it
associated surface} of $G$.

And there is a 1-to-1 correspondence between associate
surfaces and embeddings of a graph, hence an embedding of a graph
on a surface can be represented by an associate surface of it.

 After we obtain an associated surface, we aim to determine its genus, i.e., the genus of the surface on which this graph is embedded. At this stage, the following topological equivalences $\sim$ on surfaces will be useful.
\\TT 1: $Aaa^{-}B\sim AB$ where $a\notin AB$,
\\TT 2: $AxByCx^{-}Dy^{-}\sim ADCBxyx^{-}y^{-}.$\vspace{1mm}
 \\In TT 1 and TT 2, $A,B,C$ and $D$ are all linearly ordered letters and allowed to be empty. In \cite{Ringel74}, TT 1 is called simple normalizations and TT 2 is called handle normalization.

\subsection{The associate surfaces of $D_n$, $D_{i,j}$ and $W_n$} \label{subsec2:2}
The dipole $D_{n}$ consists of two vertices joined by $n$ edges.
The tripolar graph $D_{i,j}$ is the graph whose vertex set consists of three vertices, say $v_0, v_1, v_2$, and there are $i$ edges joining $v_1$ to $v_0$ and $j$ edges joining $v_2$ to $v_0$. So $D_{i,j}$ can be seen as a vertex-amalgamation of $D_i$ and $D_j$ at a vertex.
In the following, we will show the joint trees and forms of associate surfaces of $D_n$, $D_{i,j}$ and $W_n$, respectively.

{\bf The Dipole $D_n$.} Suppose that the vertex set of $D_{n} ~(n\geq 1)$ is $V(D_{n})=\{v_1,v_2\}$ and the edge set is $E(D_n)=\{a_0, a_1,\ldots, a_{n-1}\}$. We choose $v_1a_0v_2$ as the spanning tree, then the joint tree of $D_n$ is shown in Figure \ref{f1} (left) and the associate surface forms as $(AB)$ in which $A$ is a permutation of $n-1$ letters $a_1,\ldots, a_{n-1}$ and $B$ is a permutation of $n-1$ letters $a_1^-,\ldots, a_{n-1}^-$. In Figure \ref{f1} (right), we give a joint tree of $D_4$, when we read the lettered semi-edges in clockwise order, we get an associate surface of $D_4$, that is $a_3a_2a_1a_2^-a_3^-a_1^-$. Using TT 2 and TT 1 successively, we have
$$a_3a_2a_1a_2^-a_3^-a_1^-\sim a_2a_1a_2^-a_1^-a_3a_3^-\sim a_2a_1a_2^-a_1^- \sim \mathbb{S}_1,$$ which corresponds to an embedding of $D_4$ on the torus.
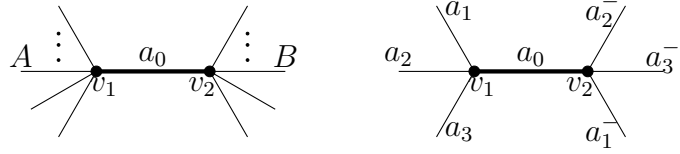
\begin{figure}[h]
\centering
\begin{tikzpicture}
\begin{scope}
\draw [ultra thick](0,0)--(1.5,0);
\draw{(-120:1) -- (0, 0)(-150:1) -- (0, 0)(-180:1) -- (0, 0)(-240:1) -- (0, 0)};
\draw{[xshift=1.5cm](60:1) -- (0, 0)(-60:1) -- (0, 0)(0:1) -- (0, 0)(-30:1) -- (0, 0)};
\filldraw [black] (-0.5,0.5) circle (0.5pt);
\filldraw [black] (-0.5,0.35) circle (0.5pt);
\filldraw [black] (-0.5,0.15) circle (0.5pt);
\filldraw [black] (2,0.5) circle (0.5pt);
\filldraw [black] (2,0.35) circle (0.5pt);
\filldraw [black] (2,0.15) circle (0.5pt);
\filldraw [black] (0,0) circle (2pt);
\filldraw [black] (1.5,0) circle (2pt);	
\node at (0.1,-0.2) {$v_1$};\node at (1.4,-0.2) {$v_2$};\node at (0.75,0.2) {$a_0$};
\node at (-1,0.2) {$A$};\node at (2.5,0.2) {$B$};		
\end{scope}

\begin{scope}[xshift=5cm]
\filldraw [black] (0,0) circle (2pt);
\filldraw [black] (1.5,0) circle (2pt);	
\draw [ultra thick](0,0)--(1.5,0);
\draw{(-120:1) -- (0, 0)(-180:1) -- (0, 0)(-240:1) -- (0, 0)};
\draw{[xshift=1.5cm](60:1) -- (0, 0)(-60:1) -- (0, 0)(0:1) -- (0, 0)};
\node at (0.1,-0.2) {$v_1$};\node at (1.4,-0.2) {$v_2$};\node at (0.75,0.2) {$a_0$};
\node at (-1,0.2) {$a_2$};\node at (2.5,0.2) {$a_3^-$};	
\node at (-0.2,0.8) {$a_1$};\node at (1.7,0.8) {$a_2^-$};
\node at (-0.2,-0.8) {$a_3$};\node at (1.7,-0.8) {$a_1^-$};
\end{scope}
\end{tikzpicture}
\caption{The joint tree of $D_n$ (left) and a joint tree of $D_4$ (right)}\label{f1}
\end{figure}

{\bf The tripolar graph $D_{i,j}$.} Suppose that the vertex set of $D_{i,j}~ (i,j\geq 1)$ is $V(D_{i,j})=\{v_0, v_1, v_2\}$ and the edge set is $E(D_{i,j})=\{a_0, a_1,\ldots, a_{i-1}, b_0, b_1,\ldots b_{j-1}\}$ in which the endpoints of $a_m ~(0\leq m\leq i-1)$ are $v_1$ and $v_0$, the endpoints of $b_t ~(0\leq t\leq j-1)$ are $v_2$ and $v_0$. We choose $v_1a_0v_0b_0v_2$ as the spanning tree, then the joint tree of $D_{i,j}$ is shown in Figure \ref{f2} (left) and the associate surface forms as $(ADBC)$ in which $A$ is a permutation of $i-1$ letters $a_1,\ldots, a_{i-1}$, $B$ is a permutation of $j-1$ letters $b_1,\ldots, b_{j-1}$ and $DC$ is the permutation of $a_1^-,\ldots, a_{i-1}^-, b_1^-\ldots, b_{j-1}^- $. In Figure \ref{f2} (right), we give a joint tree of $D_{3,2}$, whose corresponding associate surface is $a_1a_2a_1^-b_1^-b_1a_2^-$. By using TT 1 for once, we have $$a_1a_2a_1^-b_1^-b_1a_2^-\sim a_1a_2a_1^-a_2^-\sim \mathbb{S}_1,$$ so it represents an embedding of $D_{3,2}$ on the torus.

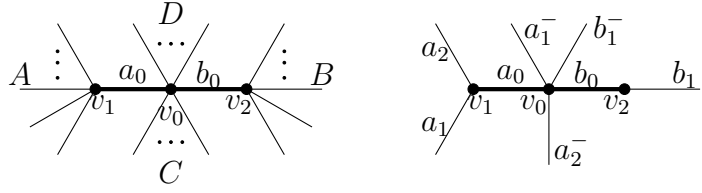
\begin{figure}[h]
\centering
\begin{tikzpicture}
\begin{scope}
\draw [ultra thick](0,0)--(2,0);
\draw{(-120:1) -- (0, 0)(-150:1) -- (0, 0)(-180:1) -- (0, 0)(-240:1) -- (0, 0)};
\draw{[xshift=2cm](60:1) -- (0, 0)(-60:1) -- (0, 0)(0:1) -- (0, 0)(-30:1) -- (0, 0)};
\draw{[xshift=1cm](-120:1) -- (0, 0)(60:1) -- (0, 0)(120:1) -- (0, 0)(-60:1) -- (0, 0)};
\filldraw [black] (-0.5,0.5) circle (0.5pt);
\filldraw [black] (-0.5,0.35) circle (0.5pt);
\filldraw [black] (-0.5,0.15) circle (0.5pt);
\filldraw [black] (2.5,0.5) circle (0.5pt);
\filldraw [black] (2.5,0.35) circle (0.5pt);
\filldraw [black] (2.5,0.15) circle (0.5pt);

\filldraw [black] (0.85,0.6) circle (0.5pt);
\filldraw [black] (1,0.6) circle (0.5pt);
\filldraw [black] (1.15,0.6) circle (0.5pt);
\filldraw [black] (0.85,-0.7) circle (0.5pt);
\filldraw [black] (1,-0.7) circle (0.5pt);
\filldraw [black] (1.15,-0.7) circle (0.5pt);

\filldraw [black] (0,0) circle (2pt);
\filldraw [black] (1,0) circle (2pt);	
\filldraw [black] (2,0) circle (2pt);	
\node at (0.1,-0.2) {$v_1$};\node at (1.9,-0.2) {$v_2$};\node at (1,-0.35) {$v_0$};
\node at (0.5,0.2) {$a_0$};\node at (1.5,0.2) {$b_0$};
\node at (-1,0.2) {$A$};\node at (3,0.2) {$B$};		
\node at (1,1) {$D$};\node at (1,-1.1) {$C$};	
\end{scope}

\begin{scope}[xshift=5cm]
\draw [ultra thick](0,0)--(2,0);
\draw{(-120:1) -- (0, 0)(-240:1) -- (0, 0)};
\draw{[xshift=2cm](0:1) -- (0, 0)};
\draw{[xshift=1cm](120:1) -- (0, 0)(60:1) -- (0, 0)(-90:1) -- (0, 0)};
\filldraw [black] (0,0) circle (2pt);
\filldraw [black] (1,0) circle (2pt);	
\filldraw [black] (2,0) circle (2pt);	
\node at (0.1,-0.2) {$v_1$};\node at (1.9,-0.2) {$v_2$};\node at (0.8,-0.2) {$v_0$};
\node at (0.5,0.2) {$a_0$};\node at (1.5,0.2) {$b_0$};
\node at (-0.5,0.5) {$a_2$};\node at (-0.5,-0.5) {$a_1$};		
\node at (0.9,0.8) {$a_1^-$};\node at (1.8,0.8) {$b_1^-$};\node at (2.8,0.2) {$b_1$};\node at (1.3,-0.8) {$a_2^-$};	
\end{scope}
\end{tikzpicture}
\caption{The joint tree of $D_{i,j}$ (left) and a joint tree of $D_{3,2}$ (right)}\label{f2}
\end{figure}

{\bf The wheel graph $W_n$.} Suppose the the vertex set of $W_n$ is $V(W_n)=\{v_0, v_1,\ldots, v_n\}$ in which $(v_1v_2\cdots v_n)$ is the $n$-cycle, $v_0$ joins to each $v_i$ $(1\leq i\leq n)$. We choose $v_1v_2\cdots v_nv_0$ as the spanning tree, label cotree edge $v_1v_n$ by $a_0$, label cotree edge $v_0v_{i}$ by $a_i$ for $1\leq i\leq n-1$. According to the places of $a_0$ and $a_0^-$ in the joint trees, we can classify all joint trees into two cases as shown in Figure \ref{f3}. And the associate surfaces form as $(a_0Ba_0^{-}AC)$ and $(a_0BAa_0^{-}C)$ in which $A$ is a permutation of $n-1$ letters $a_1,\ldots, a_{n-1}$, $BC$ is a permutation of $n-1$ letters $a_1^-,\ldots, a_{n-1}^-$.
It is easy to see that the two cases have symmetry,  so we only need to consider one of them.

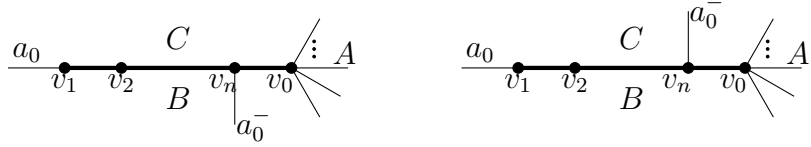
\begin{figure}[h]
\centering
\begin{tikzpicture}
\begin{scope}
\draw [ultra thick](0,0)--(3,0);
\draw{(-180:0.75) -- (0, 0)};
\draw{[xshift=3cm](60:0.75) -- (0, 0)(0:0.75) -- (0, 0)(-30:0.75) -- (0, 0)(-60:0.75) -- (0, 0)};
\draw{[xshift=2.25cm](-90:0.75) -- (0, 0)};

\filldraw [black] (3.3,0.35) circle (0.5pt);
\filldraw [black] (3.3,0.25) circle (0.5pt);
\filldraw [black] (3.3,0.15) circle (0.5pt);

\filldraw [black] (0,0) circle (2pt);
\filldraw [black] (0.75,0) circle (2pt);	
\filldraw [black] (2.25,0) circle (2pt);
\filldraw [black] (3,0) circle (2pt);

\node at (0,-0.2) {$v_1$};\node at (0.75,-0.2) {$v_2$};\node at (2.1,-0.2) {$v_n$};\node at (2.85,-0.2) {$v_0$};
\node at (-0.5,0.2) {$a_0$}; \node at (1.5,0.4) {$C$};\node at (1.5,-0.4) {$B$};\node at (3.7,0.2) {$A$};\node at (2.5,-0.75) {$a_0^-$};
	
\end{scope}

\begin{scope}[xshift=6cm]
\draw [ultra thick](0,0)--(3,0);
\draw{(-180:0.75) -- (0, 0)};
\draw{[xshift=3cm](60:0.75) -- (0, 0)(0:0.75) -- (0, 0)(-30:0.75) -- (0, 0)(-60:0.75) -- (0, 0)};
\draw{[xshift=2.25cm](90:0.75) -- (0, 0)};

\filldraw [black] (3.3,0.35) circle (0.5pt);
\filldraw [black] (3.3,0.25) circle (0.5pt);
\filldraw [black] (3.3,0.15) circle (0.5pt);

\filldraw [black] (0,0) circle (2pt);
\filldraw [black] (0.75,0) circle (2pt);	
\filldraw [black] (2.25,0) circle (2pt);
\filldraw [black] (3,0) circle (2pt);

\node at (0,-0.2) {$v_1$};\node at (0.75,-0.2) {$v_2$};\node at (2.1,-0.2) {$v_n$};\node at (2.85,-0.2) {$v_0$};
\node at (-0.5,0.2) {$a_0$}; \node at (1.5,0.4) {$C$};\node at (1.5,-0.4) {$B$};\node at (3.7,0.2) {$A$};\node at (2.5,0.75) {$a_0^-$};	
\end{scope}
\end{tikzpicture}
\caption{The joint trees of $W_n$}\label{f3}
\end{figure}

\subsection{Reduction formula of $\Gamma_{W_n}(x)$}
In this subsection, by analyzing the relations among the associated surfaces of graphs $W_n$, $D_n$, and $D_{n,n-i}$, we
establish a reduction formula of $\Gamma_{W_n}(x)$ in term of $\Gamma_{D_n}(x)$ and
$\Gamma_{D_{i+1,n-i-1}}(x)$. We denote the number of surfaces of genus $k$ in a surface set $S$ by $g_k(S)$.

\begin{lemma}\label{le2:1}
For the genus polynomial of the wheel graph $W_n$~$(n\geq 3)$, $$\Gamma_{W_n}(x)= \frac{2}{(n-1)!}\Gamma_{D_{n}}(x)+2x\sum\limits_{i=0}^{n-2}\binom{n-1}{i}\frac{1}{i!(n-2-i)!}\Gamma_{D_{i+1,n-1-i}}(x).$$
\end{lemma}

\begin{proof}
We define $n+2$ surface sets $M_0,\ldots, M_{n+1}$ as follows, in which $A$ is a permutation of $n-1$ letters $a_1^-,\ldots, a_{n-1}^-$ and
$CD$ is a permutation of $n-2$ letters $a_1^-,\ldots, a_{n-2}^-$.
$$M_{0}=\{a_{0}a_{1}\cdots
a_{n-1}a^{-}_{0}A\},$$ $$M_{i}=\{a_{i} \cdots
a_{1}a_{0}a_{i+1} \cdots
a_{n-1}a^{-}_{0}A\}~(1\leq i\leq n-2),$$ $$M_{n-1}=\{a_{n-1}\cdots a_{1}a_{0}a^{-}_{0}A\},$$
$$M_n=\{a_1a_2\cdots a_{n-1}A\},$$
$$M_{n+1}=\{a_1\cdots a_{i-1}C a_i\cdots a_{n-2}D\}.$$

According to the associate surfaces of $D_n$, $D_{i,n-i}$ and $W_n$ and the symmetry,  we have
\begin{eqnarray}\label{e1}
g_k(D_n)=(n-1)!g_k(M_n),\end{eqnarray}
\begin{eqnarray}\label{e2}
g_k(D_{i,n-i})=(i-1)!(n-i-1)!g_k(M_{n+1}),\end{eqnarray}
\begin{eqnarray}\label{e3} g_k(W_n)=\sum\limits_{i=0}^{n-1}2\binom{n-1}{i}g_k(M_i).\end{eqnarray}

Now we consider $M_0$, let $A=A_1a_{n-1}^-A_2$, then from TT2
\begin{eqnarray*}
a_{0}a_{1}\cdots a_{n-1}a^{-}_{0}A&=&a_{0}a_{1}\cdots a_{n-1}a^{-}_{0}A_1a_{n-1}^-A_2\\
&\sim & a_1\cdots a_{n-2}A_2A_1a_0a_{n-1}a_0^-a_{n-1}^-.
\end{eqnarray*}
So we have for $k\geq 1$,
\begin{eqnarray}\label{e4} g_k(M_0)=\frac{n-1}{(n-2)!}g_{k-1}(D_{n-1}),\end{eqnarray}
and $g_0(M_0)=0.$

Then we consider $M_i$,$(1\leq i\leq n-2)$, let $A=A_1a_{n-1}^-A_2$, then from TT2
\begin{eqnarray*}
a_{i} \cdots a_{1}a_{0}a_{i+1} \cdots a_{n-1}a^{-}_{0}A&=&a_{i} \cdots a_{1}a_{0}a_{i+1} \cdots a_{n-1}a^{-}_{0}A_1a_{n-1}^-A_2\\
&\sim & a_i\cdots a_1A_1a_{i+1}\cdots a_{n-2}A_2a_0a_{n-1}a_0^-a_{n-1}^-.
\end{eqnarray*}
So we have for $k\geq 1$,
\begin{eqnarray}\label{e5}
g_k(M_i)=\frac{1}{i!(n-i-2)!}g_{k-1}(D_{i+1,n-1-i}),\end{eqnarray}
and $g_0(M_i)=0.$

Last we consider $M_{n-1}$, from TT1
$$a_{n-1}\cdots a_{1}a_{0}a^{-}_{0}A\sim a_{n-1}\cdots a_{1}A.$$
So we have
\begin{eqnarray}\label{e6}
g_k(M_{n-1})=g_k(M_n)=\frac{1}{(n-1)!}g_{k}(D_{n}),\end{eqnarray}
and it is easy to get that $g_0(M_{n-1})=1$, i.e., when $A=a_1^-\cdots a_{n-1}^-$.

Combining \eqref{e1}-\eqref{e6}, we have that
\begin{eqnarray}\label{e7}
g_k(W_n)
&=& 2\sum\limits_{i=0}^{n-2}\binom{n-1}{i}g_k(M_i)+2\binom{n-1}{n-1}g_k(M_{n-1})\nonumber\\
&=& \left\{\begin{array}{cc}
              2\sum\limits_{i=0}^{n-2}\binom{n-1}{i}\frac{1}{i!(n-2-i)!}g_{k-1}(D_{i+1,n-1-i})+\frac{2}{(n-1)!}g_k(D_{n}),& \mbox{when $k\geq 1;$}\\
               2& \mbox{when $k=0$.}
              \end{array}\right.
\end{eqnarray}

 From the definition of the genus polynomial, we obtain the following expansions:
$$\Gamma_{W_n}(x) = 2 + \sum_{k\ge 1} g_k(W_n)x^k,$$
$$\Gamma_{D_n}(x) = (n-1)! + \sum_{k\ge 1} g_k(D_n)x^k,$$
$$\Gamma_{D_{i+1,n-1-i}}(x) = \sum_{k\ge 0} g_k(D_{i+1,n-1-i})x^k = \sum_{k\ge 1} g_{k-1}(D_{i+1,n-1-i})x^{k-1}.$$
Consequently,
\[
\sum_{k\ge 1} g_k(D_n)x^k = \Gamma_{D_n}(x) - (n-1)!,
\qquad\mbox{and}\qquad
\sum_{k\ge 1} g_{k-1}(D_{i+1,n-1-i})x^k = x\,\Gamma_{D_{i+1,n-1-i}}(x).
\]
Hence, the reduction formula for the genus polynomial of the wheel graph $W_n$ is given by

\begin{eqnarray*}
\Gamma_{W_n}(x)&=&2+\frac{2}{(n-1)!}(\Gamma_{D_{n}}(x)-(n-1)!)+2\sum\limits_{i=0}^{n-2}\binom{n-1}{i}\frac{1}{i!(n-2-i)!}x\Gamma_{D_{i+1,n-1-i}}(x)\notag\\
&=& \frac{2}{(n-1)!}\Gamma_{D_{n}}(x)+2x\sum\limits_{i=0}^{n-2}\binom{n-1}{i}\frac{1}{i!(n-2-i)!}\Gamma_{D_{i+1,n-1-i}}(x).
\end{eqnarray*}
The proof is complete.
\end{proof}

\subsection{Class algebra of the symmetric group}
In this subsection, we present some algebraic preliminaries required for the proof of Theorem \ref{thm:distribution}

Recall that a partition $\lambda$ of $n$ may be written in frequency notation as
\[
\lambda=(1^{\lambda_1}2^{\lambda_2}\cdots n^{\lambda_n}),
\qquad
n=\lambda_1+2\lambda_2+\cdots+n\lambda_n,
\]
where $\lambda_j\in\mathbb{Z}_{\geq 0}$. And we write $\lambda\vdash n$ to indicate that $\lambda$ is a partition of $n$, we also use it for conjugacy classes of the symmetric group $S_n$, because there is a natural one-to-one correspondence between  partitions of $n$ and conjugacy classes of $S_n$.

To a partition $\lambda\vdash n$, we associate a \emph{Young diagram} with $\lambda_i$ boxes in the $i$th row, the rows of boxes left-justified. We define a \emph{tableau} on a given Young diagram to be a numbering of the boxes by the integers $1, 2, \dots, n$, and we will call it \emph{standard} if the rows and columns are increasing sequences.

For a partition
\(
\mu=(1^{m_1}2^{m_2}\cdots n^{m_n})\vdash n,
\)
define
\(
z_\mu=\prod_{j=1}^{n}j^{m_j}m_j!.
\)
Thus, the conjugacy class of cycle type $\mu$ has cardinality
$
\frac{n!}{z_\mu}.
$
We use $C_\mu$ for the conjugacy class of cycle type $\mu$ and, by the usual
abuse of notation, also for the corresponding class sum in the group algebra
$\mathbb{C}[S_n]$.

Let $\chi_\mu^\theta$ be the value of the irreducible character indexed by
$\theta\vdash n$ on the conjugacy class $C_\mu$. Define
\[
f^\theta=\chi_{(1^n)}^\theta.
\]
Hence, $f^\theta$ is the degree of the irreducible character $\chi^\theta$,
equivalently the dimension of the irreducible $S_n$-module indexed by
$\theta$. It is also the number of standard Young tableaux of shape $\theta$.

For each partition $\theta\vdash n$, let $F^\theta$ denote the primitive
central idempotent of $\mathbb{C}[S_n]$ corresponding to $\chi^\theta$:
\[
F^\theta
=
\frac{f^\theta}{n!}
\sum_{\sigma\in S_n}\chi^\theta(\sigma^{-1})\sigma
=
\frac{f^\theta}{n!}
\sum_{\mu\vdash n}\chi_\mu^\theta C_\mu.
\]
The elements $F^\theta$ satisfy
\[
F^\theta F^\nu=\delta_{\theta\nu}F^\theta,
\qquad
\sum_{\theta\vdash n}F^\theta=1_n,
\]
 and
\[
[C_{(1^n)}]F^\theta=\frac{(f^\theta)^2}{n!},
\]
where $[C_{(1^n)}] F^\theta$ denotes the coefficient of
$C_{(1^n)}=1_n$ in $F^\theta \in\mathbb{C}[S_n]$.

For a permutation $\sigma\in S_n$, let $\normcycles{\sigma}$ denote the
number of cycles of $\sigma$, including fixed points. If $\mu\vdash n$, then
$\normcycles{\mu}$ denotes the number of parts of $\mu$, equivalently the
number of cycles of a permutation of cycle type $\mu$.

Define
\[
H_\theta(x)
=
\sum_{\mu\vdash n}
\frac{n!\chi_\mu^\theta}{f^\theta z_\mu}
x^{\,n-\normcycles{\mu}}.
\]
Equivalently, if $c(u)$ denotes the content of a box $u$ in the Young diagram
of $\theta$, then
\[
H_\theta(x)=\prod_{u\in\theta}\bigl(1+c(u)x\bigr).
\]

\subsection{Proof of Theorem \ref{thm:distribution}}
In this subsection, we compute $\Gamma_{D_{i,j}}(x)$ by using
the character approach and derive $\Gamma_{D_{n}}(x)$
and $\Gamma_{W_n}(x)$ by the relations of the three.

\begin{proof}[\bf{Proof of Theorem \ref{thm:distribution}}]

For the tripolar graph $D_{i,n-1}$, let $v_{0},~v_1$, and $v_2$ be the vertices of it as defined in Subsection \ref{subsec2:2}.
Suppose $\rho_{v_0},\rho_{v_1},$ and $\rho_{v_2}$ are the vertex rotations of $v_0$, $v_1$ and $v_2$ respectively. Let $1_n$ be the identity element of $S_n.$ It is well known that any orientable embedding of
$D_{i,n-i}$ can be encoded by the factorization
 $\rho_{v_0}\rho_{v_1}\rho_{v_2}\varrho=1_n.$
 From Euler's formula, we have
\[
3+\normcycles{\rho_{v_0}\rho_{v_1}\rho_{v_2}}
=n+2-2g,
\]
equivalently,
\[
2g
=n-1-\normcycles{\rho_{v_0}\rho_{v_1}\rho_{v_2}},
\]
where $g$ is the genus of the embedded surface.

We have the following expression $\Gamma_{D_{i,n-i}}(x^2)$ of the genus polynomial of $D_{i,n-i}$: $$\Gamma_{D_{i,n-i}}(x^2)=\dsum_{\substack{\rho_{v_0}\rho_{v_1}\rho_{v_2}\varrho=1_n\\\rho_{v_0}\in C_{(n)},\rho_{v_1}\rho_{v_2}\in C_{(i,n-i})}} \frac{i!(n-i)!}{n!} x^{n-1-\parallel\rho_{v_0}\rho_{v_1}\rho_{v_2}\parallel}.$$

Using Frobenius formula (see, e.g., \cite{Ja78}), we rewrite the expression of $\Gamma_{D_{i,n-i}}(x^2)$ above as

\begin{eqnarray*}\Gamma_{D_{i,n-i}}(x^2)&=&\frac{i!(n-i)!}{n!}[C_{(1^n)}]\dsum_{\mu\vdash n}C_{(n)}C_{(i,n-i)}C_{\mu}x^{n-1-\parallel\mu\parallel}\notag \\
&=&\frac{i!(n-i)!}{n!}[C_{(1^n)}]\dsum_{\mu\vdash n}\frac{(n!)^3}{z_{(n)}z_{(i,n-i)}z_{\mu}}\dsum_{\theta\vdash n}\frac{\chi_{(n)}^{\theta}\chi_{(i,n-i)}^{\theta}\chi_{\mu}^{\theta}}{(f^{\theta})^3}F^{\theta}x^{n-1-\parallel\mu\parallel}\notag\\
&=&\frac{i!(n-i)!}{n!}\dsum_{\mu\vdash n}\frac{(n!)^2}{z_{(n)}z_{(i,n-i)}z_{\mu}}\dsum_{\theta\vdash n}\frac{\chi_{(n)}^{\theta}\chi_{(i,n-i)}^{\theta}\chi_{\mu}^{\theta}}{f^{\theta}}x^{n-1-\parallel\mu\parallel}\notag\\
&=&\frac{i!(n-i)!}{n!}\dsum_{\theta\vdash n}\frac{n!\chi_{(n)}^{\theta}\chi_{(i,n-i)}^{\theta}}{z_{(n)}z_{(i,n-i)}}\dsum_{\mu\vdash n}\frac{n!\chi_{\mu}^{\theta}}{f^{\theta}z_{\mu}}x^{n-1-\parallel\mu\parallel}\notag\\
&=&\frac{i!(n-i)!}{n!}\dsum_{\theta\vdash n}\frac{n!\chi_{(n)}^{\theta}\chi_{(i,n-i)}^{\theta}}{z_{(n)}z_{(i,n-i)}}\frac{H_{\theta}(x)}{x}.
\end{eqnarray*}

Since the character values  (see, e.g., \cite{Ja87}) are given by
\[
\chi_{(n)}^\tau =
\begin{cases}
(-1)^i, & \text{if } \tau=(1^i,n-i),\quad i=0,1,\dots,n-1,\\[2pt]
0, & \text{otherwise},
\end{cases}
\]
and
\[
\chi_{(m,n-m)}^{(1^i,n-i)} =
\begin{cases}
(-1)^i, & \text{if } 0\le i < m,\\[2pt]
0, & \text{if } m\le i < n-m,\\[2pt]
(-1)^{i-1}, & \text{if } n-m\le i < n,
\end{cases}
\]
using the character sum identity (Lemma 3.4 in \cite{JV90}, or Lemma 1 in \cite{Fang14}),
we obtain that, for \(1\le i\le n-1\), the following expression for the genus polynomial of \(D_{i,n-i}\):
\[
\begin{aligned}
\Gamma_{D_{i,n-i}}(x^2)
&= \frac{(i-1)!(n-i-1)!}{n}
\sum_{k=0}^{n-1} (-1)^k \chi_{(i,n-i)}^{(1^k,n-k)}
\frac{H_{(1^k,n-k)}(x)}{x} \nonumber \\
&= \frac{(i-1)!(n-i-1)!}{n}
\left(
\sum_{k=0}^{i-1} \frac{H_{(1^k,n-k)}(x)}{x}
-
\sum_{k=n-i}^{n-1} \frac{H_{(1^k,n-k)}(x)}{x}
\right) \\
&= \frac{(i-1)!(n-i-1)!}{n}
\left(
\sum_{k=0}^{i-1}
\frac{\prod_{j=k-n+1}^{k}(1-jx)}{x}
-
\sum_{k=n-i}^{n-1}
\frac{\prod_{j=k-n+1}^{k}(1-jx)}{x}
\right) \\
&= \frac{(-1)^n (i-1)!(n-i-1)!}{n} x^{n-1}
\left(
\sum_{k=0}^{i-1}
\left(k-n+1-\frac1x\right)^{\overline{n}}
-
\sum_{k=n-i}^{n-1}
\left(k-n+1-\frac1x\right)^{\overline{n}}
\right).
\end{aligned}
\]

For simplicity, we write $F(k,x)$ for
$k-n+1-\frac{1}{\sqrt{x}}$. Hence, the genus polynomial of the tripolar graph \(D_{i,n-i}\)
\begin{eqnarray}\label{e8}
\Gamma_{D_{i,n-i}}(x)=\frac{(-1)^n (i-1)!(n-i-1)!}{n}x^{\frac{n-1}{2}}
\left(
\sum_{k=0}^{i-1}
\left(F(k,x)\right)^{\overline{n}}
-
\sum_{k=n-i}^{n-1}
\left(F(k,x)\right)^{\overline{n}}
\right)
\end{eqnarray}

For the dipole graph $D_n$, since $n\Gamma_{D_n}(x)=\Gamma_{D_{1,n}}(x)$, one can compute that its genus polynomial is
\begin{eqnarray}\label{e9}
\Gamma_{D_n}(x)=\frac{2(n-1)!}{n(n+1)}
\sum_{m=0}^{\lfloor n/2\rfloor}
\left[ {n+1 \atop n-2m} \right]x^m.
\end{eqnarray}

Combining \eqref{e8}, \eqref{e9} and Lemma \ref{le2:1}, we obtain the following explicit formula for the genus polynomial of the wheel graph:

\begin{eqnarray}\label{e10}
\Gamma_{W_n}(x) &=& \frac{4}{n(n+1)} \sum_{m=0}^{\lfloor n/2 \rfloor} \begin{bmatrix} n+1 \\ n-2m \end{bmatrix} x^m \nonumber\\
&&\quad + \frac{2(-1)^n}{n} x^{(n+1)/2} \sum_{i=0}^{n-2} \binom{n-1}{i} \left(
\sum_{k=0}^{i-1}
\left(F(k,x)\right)^{\overline{n}}
-
\sum_{k=n-i}^{n-1}
\left(F(k,x)\right)^{\overline{n}}
\right)\nonumber\\
&=& \frac{4}{n(n+1)} \sum_{m=0}^{\lfloor n/2 \rfloor} \begin{bmatrix} n+1 \\ n-2m \end{bmatrix} x^m
+ \frac{2(-1)^n}{n} x^{(n+1)/2} \sum_{k=0}^{n-1} c_k \left(F(k,x)\right)^{\overline{n}},
\end{eqnarray}
in which
\[
c_k =
\begin{cases}
2^{n-1} - 1, & k = 0, \\[4pt]
\displaystyle \sum_{j=k+1}^{n-1-k} \binom{n-1}{j}, & 1 \le k \le n-2, \\[4pt]
-\bigl(2^{n-1} - 1\bigr), & k = n-1.
\end{cases}
\]
The proof is complete.
\end{proof}

\begin{example} Using Theorem \ref{thm:distribution}, we list the genus polynomial of $W_{n}$ for $3\leq n\leq 7$ as follows.
\begin{flalign*}
\Gamma_{W_3}(x) &= 2 + 14x,&\\
\Gamma_{W_4}(x) &= 2+58x+36x^2,&\\
\Gamma_{W_5}(x) &= 2+190x+576x^2,&\\
\Gamma_{W_6}(x) &= 2+550x+4968x^2+2160x^3,&\\
\Gamma_{W_7}(x) &= 2+1484x+31178x^2+59496x^3.
\end{flalign*}
\end{example}

\begin{remark} In \cite{Chen2018}, Chen, Gross, and Mansour obtained the region distribution of $W_n$.
 By applying Euler's formula, one can further derive from it the genus distribution of $W_n$.
In this paper, We use a different method from theirs to directly derive the genus distribution. Our expression is more convenient for the computations that follow. It is also worth noting that
our paper presents the first attempt to combine the methods of joint trees and characters to compute the genus polynomial of graphs. Perhaps this combination can solve genus distribution problems for more classes of graphs.
\end{remark}

\section{Unimodality of Genus Distributions of Wheels}\label{Sec:3}
In this section, we first express the auxiliary function $Q_n(t)$ constructed herein using the shift operator. Subsequently, by proving that all roots of
$Q_n(t)$ are purely imaginary, we conclude that all roots of $\Gamma_{W_n}(x)$ are negative real numbers, thereby establishing the log-concavity and unimodality of its coefficients.

\subsection{Shift operator representation}
We define \[
Q_n(t)=\frac{1}{2} t^{n+1}\Gamma_{W_n}(t^{-2}),
\] and \[F_{n}(t)=t^{\overline{n+1}}=t(t+1)\cdots(t+n).\] For \(0\le k\le n-1\), we set
\[
P_{n,k}(t)=\prod_{j=k-n+1}^{k}(t-j).
\]
Let \( E \) denote the backward shift operator on the polynomial space \( \mathbb{C}[t] \) defined by
\[ Ef(t) = f(t-1). \]

\begin{proposition}\label{pro3:1}
Let the notation be as above. The following identities hold.
\\(i) \begin{eqnarray}\label{e11} (1-E^n)F_{n}(t)= 2 \sum_{m=0}^{\lfloor n/2 \rfloor} \left[ {n+1 \atop n-2m} \right]  t^{n-2m};\end{eqnarray}
(ii) For $0\leq k\leq n-1,$ \begin{eqnarray}\label{e12}(E^k-E^{k+1})F_{n}(t)=(n+1)P_{n,k}(t);\end{eqnarray}
(iii) \begin{eqnarray}\label{e13}
t(1-E^n)F_{n}(t)= \left( 1 + (n + 2) \sum_{k=1}^{n-1} E^k + E^n \right) F_n(t).\end{eqnarray}
\end{proposition}

\begin{proof} For the item (i), it is well known that  signless Stirling numbers of the first kind satisfies
\[F_n(t) = \sum_{j=0}^{n+1} \left[ {n+1 \atop j} \right]t^j \] and
\[E^nF_{n}(t)=F_n(t - n) = t(t - 1) \cdots (t - n) = \sum_{j=0}^{n+1} (-1)^{n+1-j} \left[ {n+1 \atop j} \right] t^j.\]
Subtracting the above two identities gives
$$(1-E^n)F_{n}(t)=2\sum_{\substack{j=0 \\ n+1-j ~\text{is odd}}}^{n+1}\left[ {n+1 \atop j} \right]t^j =2 \sum_{m=0}^{\lfloor n/2 \rfloor} \left[ {n+1 \atop n-2m} \right]  t^{n-2m}.$$

For the item (ii),
\begin{eqnarray*}
(E^k-E^{k+1})F_{n}(t)&=&\prod_{j=k-n}^{k}(t-j)-\prod_{j=k-n+1}^{k+1}(t-j)\\
&=&(t-k+n)P_{n,k}-(t-k-1)P_{n,k}\\
&=&(n+1)P_{n,k}.
\end{eqnarray*}

For the item (iii), from the definition of $F_n$, we have
$$(t-k-1)F_{n}(t-k)=(t-k+n)F_n(t-k-1).$$
By using this,
\begin{eqnarray}\label{e14}
&&(t+n+1)F_{n}(t)-(t+1)F_{n}(t-n)\nonumber\\
&=&\sum_{k=0}^{n-1}\big((t-k+n+1)F_{n}(t-k)-(t-k+n)F_{n}(t-k-1)\big)\nonumber\\
&=&\sum_{k=0}^{n-1}\big((t-k+n+1)F_{n}(t-k)-(t-k-1)F_{n}(t-k)\big)\nonumber\\
&=&\sum_{k=0}^{n-1}(n+2)F_{n}(t-k)
\end{eqnarray}
From \eqref{e14}, we have
$$t\left(F_{n}(t)-F_{n}(t-n)\right)=F_{n}(t)+(n+2)\sum_{k=1}^{n-1}F_{n}(t-k)+F_{n}(t-n),$$ this proves  \eqref{e13}.
\end{proof}

We are now ready to prove the following shift operator representation for the auxiliary function $Q_n(t)$.

\begin{lemma}\label{le:shift}
For every $n\geq 3$,
\[Q_n(t)=h_n(E)F_n(t),
\]
where
\[
h_n(z)=\frac{1}{n(n+1)}
\left(
2^{n-1}(1+z^n)+\sum_{k=1}^{n-1}\left(n+2-\binom{n}{k}\right)z^k
\right).
\]
\end{lemma}
\begin{proof}
Using the explicit formula \eqref{e10} for $\Gamma_{W_n}(x)$ and substituting \(x=t^{-2}\) (so that \(1/\sqrt{x}=t\)) into the definition \(Q_n(t)=\frac12 t^{n+1}\Gamma_{W_n}(t^{-2})\), we obtain
\begin{eqnarray}\label{e15}
Q_n(t)=\frac{2}{n(n+1)}
\sum_{m=0}^{\lfloor n/2\rfloor}
\begin{bmatrix}n+1\\ n-2m\end{bmatrix}t^{n+1-2m}
+\frac{(-1)^n}{n}\sum_{k=0}^{n-1}c_k(k-n+1-t)^{\overline{n}}.
\end{eqnarray}
From \eqref{e11} and \eqref{e13},
\begin{eqnarray}\label{e16}
\frac{2}{n(n+1)}\sum_{m=0}^{\lfloor n/2\rfloor}
\begin{bmatrix}n+1\\ n-2m\end{bmatrix}t^{n+1-2m}
=\frac{1}{n(n+1)} \left( 1 + (n + 2) \sum_{k=1}^{n-1} E^k + E^n \right) F_n(t).
\end{eqnarray}
From definition of $P_{n,k}(t)$ and \eqref{e12}, we have
\begin{eqnarray}\label{e17}
&&\frac{(-1)^n}{n}\sum_{k=0}^{n-1}c_k(k-n+1-t)^{\overline{n}}\nonumber\\
&=&\frac{1}{n}\sum_{k=0}^{n-1}c_kP_{n,k}(t)\nonumber\\
&=& \frac{1}{n(n+1)}\sum_{k=0}^{n-1}c_k(E^k-E^{k+1})F_n(t)\nonumber\\
&=&\frac{1}{n(n+1)}(c_0+\sum_{k=1}^{n-1}(c_k-c_{k-1})E^k-c_{n-1}E^n)F_{n}(t)
\end{eqnarray}
From the definition of \(c_k\), we have $c_0=2^{n-1}-1$, $c_{n-1}=-(2^{n-1}-1)$, and
for $1\leq k\leq n-1$, $c_k-c_{k-1}=-\binom{n}{k}$ by computation. Substituting $c_k$ into \eqref{e17} yields
\begin{eqnarray}\label{e18}
\frac{(-1)^n}{n}\sum_{k=0}^{n-1}c_k(k-n+1-t)^{\overline{n}}=\frac{1}{n(n+1)}[(2^{n-1}-1)(1+E^n)-\sum_{k=1}^{n-1}\binom{n}{k}E^k)]F_n(t)
\end{eqnarray}
Combining  \eqref{e16}, \eqref{e18} with \eqref{e15},
$$Q_n(t)=\frac{1}{n(n+1)}
\left(
2^{n-1}(1+E^n)+\sum_{k=1}^{n-1}\left(n+2-\binom{n}{k}\right)E^k
\right)F_n(t),$$
the lemma follows.
\end{proof}

\begin{remark} If we denote the coefficient of $z^k$ in $h_n(z)$ by $c_{n,k}$, then
\(h_n(z)=\sum_{k=0}^{n} c_{n,k}z^k.\) Combining it with Lemma \ref{le:shift}, we have
\begin{eqnarray}\label{e19}
Q_n(t)=\sum_{k=0}^{n}c_{n,k}F_n(t-k).
\end{eqnarray}
where
\begin{eqnarray}\label{e20}
c_{n,0}=c_{n,n}=\frac{2^{n-1}}{n(n+1)},\qquad
c_{n,k}=\frac{n+2-\binom{n}{k}}{n(n+1)}\quad (1\le k\le n-1).
\end{eqnarray}
The expression for $Q_n$ in the form of \eqref{e19} will facilitate our proof for Lemma \ref{le3:7} and calculations in Section 4.
\end{remark}

\subsection{Preliminary theorems}
In this subsection, we introduce two theorems on the locations of roots of polynomials that will be used later in our proofs.

In \cite{LakatosLosonczi2004}, a necessary condition for the zeros of self-inversive polynomials to lie on the unit circle is provided. In this paper, we will employ the special case where the polynomial in question has real coefficients, namely the self-reciprocal case.

Let \(P(z)=\sum_{k=0}^{m}a_kz^k\)~$(a_m\neq 0)$ be a polynomial of degree \(m\) with real coefficients. Then \(P\) is called \emph{self-reciprocal} if its coefficients satisfy
\[
a_k=a_{m-k}\qquad\mbox{for all }k=0,1,\dots,m.
\]

\begin{theorem}[\cite{LakatosLosonczi2004}]
\label{thm:LL}
Let \(P(z)=\sum_{k=0}^{m}a_kz^k\) be a self-reciprocal polynomial of degree $m\geq 1$. If
\[
|a_m|\ge \frac12\sum_{k=1}^{m-1}|a_k|,
\]
then all roots of \(P\) lie on the unit circle.
\end{theorem}

Another theorem to be used is a simplified version of
a theorem due to Rodriguez Villegas, which states that certain polynomials have all of their roots on a vertical line.

\begin{theorem}[\cite{RodriguezVillegas2002}]
\label{thm:RV}
Let \(U\in\mathbb{C}[z]\) be a  polynomial of degree \(e\), \(U(1)\ne 0\), and let $d\in \mathbb{N}$ with \(d\geq e+2\). Define the polynomial \(R\) by
\[
\frac{U(z)}{(1-z)^d}=\sum_{m\ge 0}R(m)z^m.
\]
If all the roots of \(U\) are on the unit circle,
then $$R(x)=(x+1)\cdots(x+d-e-1)V(x),$$ where every root of \(V\) lies on the vertical line
\[
\Re  (x)=-\frac{d-e}{2}.
\]
\end{theorem}

\subsection{The proof of  Theorem \ref{thm:main}}
In this subsection, we apply Theorems \ref{thm:LL} and \ref{thm:RV}, respectively, to determine the locations of the roots of $h_n(z)$ (defined in Lemma \ref{le:shift}) and
$Q_n(t)$. Then, we prove that all roots of
$\Gamma_{W_n}(x)$ are negative real numbers.

\begin{proposition}\label{pr3:6}
For every $n\geq 3$, all roots of $h_{n}(z)$ are on the unit circle.
\end{proposition}

\begin{proof}
Set $H_n(z)=n(n+1)h_{n}(z)$, then
$$H_{n}(z)=2^{n-1}(1+z^n)+\sum_{k=1}^{n-1}\left(n+2-\binom{n}{k}\right)z^k,$$ and the roots of $H_n(z)$ coincide with those of $h_{n}(z)$.
Because $\binom{n}{k}=\binom{n}{n-k}$, the function $H_n(z)$ is a self-reciprocal polynomial of degree $n$. Let $a_{n,k}=n+2-\binom{n}{k}$ for $1\leq k\leq n-1$, and $a_{n,n}=2^{n-1}$. From Theorem \ref{thm:LL}, it only remains to prove $ \frac12\sum_{k=1}^{n-1}|a_{n,k}|\leq 2^{n-1}$ for each $n\geq 3$, completing the proof.

When \(n=3\),
$$ \frac12\sum_{k=1}^{2}|a_{3,k}|=2\leq 4=2^{3-1}$$ holds. We now assume $n\geq 4$. For $k=1$ or $k=n-1$, we have $a_{n,1}=a_{n,n-1}=2$. For \(2\le k\le n-2\),
\(\binom{n}{k}\ge \binom{n}{2}\ge n+2\), so $a_{n,k}\leq 0$. It follows that
\[
\sum_{k=1}^{n-1}|a_{n,k}|
=4+\sum_{k=2}^{n-2}\left(\binom{n}{k}-(n+2)\right)
=2^n-n^2-n+8\le 2^n.
\]
Thus
\[
\frac12\sum_{k=1}^{n-1}|a_{n,k}|\le 2^{n-1}
\]
holds, the proposition follows.
\end{proof}

\begin{lemma}\label{le3:7}
For every $n\geq 3$, all roots of $Q_{n}(t)$ are on the imaginary axis.
\end{lemma}

\begin{proof}
We define $R_n(m)$ by
\[
\frac{h_n(z)}{(1-z)^{n+2}}=\sum_{m\ge 0}R_n(m)z^m.
\]
For every $n\geq 3$, a direct computation yields $h_n(1)=1\ne 0$. From Proposition
\ref{pr3:6} and Theorem \ref{thm:RV}, we have
$$R_n(m)=(m+1)V_{n}(m)$$
where every root of \(V_n\) lies on the vertical line
\(\Re  (m)=-1.\) Hence all roots of $R_n$ are on the vertical line
\[
\Re  (m)=-1.
\]
\textbf{Claim 1.} For every $n\geq 3$, $Q_n(t)=(n+1)!R_n(t-1)$.

Since the coefficient of $z^{m-k}$ in $1/(1-z)^{n+2}$ is $\binom{m-k+n+1}{n+1}$, and the coefficient of $z^k$ in $h_n(z)$ is $c_{n,k}$ as shown in \eqref{e20}, the coefficient of $z^m$ in
$h_n(z)/(1-z)^{n+2}$ is
\begin{eqnarray*}
R_n(m) & = & \sum_{k=0}^n c_{n,k} \binom{m-k+n+1}{n+1} \\
& = & \frac{1}{(n+1)!} \sum_{k=0}^n c_{n,k} F_n(m-k+1) \\
& = & \frac{1}{(n+1)!} Q_n(m+1),
\end{eqnarray*}
The claim follows.

Consequently, every root of \(Q_n\) satisfies \(\Re(t-1)=-1\), i.e. \(\Re t=0\). Thus all roots of \(Q_n\) are purely imaginary. The proof is complete.
\end{proof}

We are now ready to prove Theorem \ref{thm:main}.

\begin{proof}[\bf{Proof of Theorem \ref{thm:main}}]
Let \(\lambda\) be any root of \(\Gamma_{W_n}(x)\).
From \eqref{e7}, \(g_0(W_n)=2\), so we have \(\lambda\ne0\).
Choose \(t\ne0\) such that \(t^{-2}=\lambda\). Then
\[
Q_n(t)=\frac12 t^{n+1}\Gamma_{W_n}(t^{-2})=0.
\]
From Lemma \ref{le3:7}, \(t\) is purely imaginary, i.e., \(t=iy\) for some nonzero
\(y\in\mathbb{R}\). Therefore
\[
\lambda=(iy)^{-2}=-\frac1{y^2}<0.
\]
Thus every root of \(\Gamma_{W_n}(x)\) is a negative real number.

 Since \(\Gamma_{W_n}(x)\) is a polynomial with nonnegative coefficients and all real negative roots, Newton's inequalities imply that the coefficient sequence of \(\Gamma_{W_n}(x)\)  is log-concave. A log-concave sequence with no internal zeros is unimodal, so the genus distribution is unimodal. The proof is complete.
\end{proof}

\section{The Asymptotic Normality of Genus Distributions of Wheels}\label{Sec:4}
In this section, we compute the variance of the genus distribution of  \(W_n\). Subsequently, by applying a theorem of
Bender \cite{Ben73}, Theorem \ref{thm:normal} is established.

Choose an embedding uniformly at random among all orientable embeddings of a graph $G$, and let \(\gamma_n\) denote its genus. Then {\it the average genus} of $G$ is
 $$\gamma_{avg}(G)
=\E[{\gamma_n}]=\frac{\Gamma'_G(1)}{\Gamma_G(1)} = \sum_{k=0}^{\infty} \frac{k \cdot g_k(G)}{\Gamma_G(1)},$$
and the {\it variance} of the genus distribution of $G$ is
\[
\gamma_{var}(G) =\Var(\gamma_n)= \sum_{k=0}^{\infty} (k - \gamma_{avg}(G))^2 \cdot \frac{g_k(G)}{\Gamma_G(1)}.
\]
The following theorem gives asymptotic formulae for $\gamma_{avg}(W_n)$ and $\gamma_{var}(W_n)$.

\begin{theorem}
\label{thm:average}
As \(n\to\infty\),
\[
\gamma_{avg}(W_n)\sim \frac{n}{2},~~~~~\mbox{and}~~~~~\gamma_{var}(W_n)\sim \frac14\ln n.
\]
\end{theorem}

\begin{proof} Let \(Y_n=n+1-2\gamma_n\). Then we have
\begin{eqnarray}\label{e21}
\E[\gamma_n]=\frac12\left(n+1-\E[Y_n]\right) ~~~~\mbox{and}~~~~\Var(\gamma_n)=\frac14\Var(Y_n).\end{eqnarray}
One can check that $Y_n$ has the probability generating function
\[
\E[t^{Y_n}]=\frac{Q_n(t)}{Q_n(1)}.
\]
Therefore,
\begin{eqnarray}\label{e22}\E[Y_n]=\frac{Q_n'(1)}{Q_n(1)},~~~\mbox{and}~~~
\Var(Y_n)=\frac{Q_n''(1)}{Q_n(1)}
+\frac{Q_n'(1)}{Q_n(1)}
-\left(\frac{Q_n'(1)}{Q_n(1)}\right)^2.\end{eqnarray}
Now we compute $Q_n'(1)/Q_n(1)$ and $Q_n''(1)/Q_n(1)$. We use the expression for $Q_n$ in the form of \eqref{e19}, i.e., \[
Q_n(t)=\sum_{k=0}^{n}c_{n,k}F_n(t-k),
\] in which $c_{n,k}$ is shown in \eqref{e20}.  Let \(H_m=\sum_{j=1}^m1/j\), \(H_m^{(2)}=\sum_{j=1}^m1/j^2\), and \(H_0=0\).

Since \(F_n(1)=(n+1)!\)  and \(F_n(1-k)=0\)  for \(1\le k\le n\), we get
\[
Q_n(1)=c_{n,0}F_n(1)=\frac{2^{n-1}}{n(n+1)}(n+1)!=2^{n-1}(n-1)!.
\]
Then, we compute the derivative. The logarithmic derivative of \(F_n\) gives
\[
F_n'(1)=(n+1)!H_{n+1},
\]
and for \(1\le k\le n\),
\[
F_n'(1-k)=(-1)^{k-1}(k-1)!(n-k+1)!.
\]
Since
\[
Q_n'(1)=\sum_{k=0}^n c_{n,k}F_n'(1-k),
\]
we obtain
\begin{eqnarray}\label{e23}
\frac{Q_n'(1)}{Q_n(1)}
=
H_{n+1}
+
\frac{1}{2^{n-1}}
\sum_{k=1}^{n-1}
\frac{(-1)^{k-1}\left(n+2-\binom{n}{k}\right)}
{k\binom{n+1}{k}}
+
\frac{(-1)^{n-1}}{n(n+1)}.
\end{eqnarray}

Next, computing the second derivative yields
\[
F_n''(1)=(n+1)!\left(H_{n+1}^2-H_{n+1}^{(2)}\right),
\]
and for \(1\le k\le n\),
\[
F_n''(1-k)=2(-1)^{k-1}(k-1)!(n-k+1)!\left(H_{n-k+1}-H_{k-1}\right).
\]
Since
\[
Q_n''(1)=\sum_{k=0}^n c_{n,k}F_n''(1-k),
\] we obtain
\begin{eqnarray}\label{e24}
\frac{Q_n''(1)}{Q_n(1)}
&=&H_{n+1}^2-H_{n+1}^{(2)}+\frac{2(-1)^{n-1}}{n(n+1)}(1-H_{n-1})\nonumber\\
&&+\frac{1}{2^{n-2}}
\sum_{k=1}^{n-1}
\frac{(-1)^{k-1}\left(n+2-\binom{n}{k}\right)
\left(H_{n-k+1}-H_{k-1}\right)}
{k\binom{n+1}{k}}
\end{eqnarray}

In the following, we compute $\E[\gamma_n]$ and $\Var(\gamma_n)$ and carry out the asymptotics.
Combining \eqref{e21}, \eqref{e22} with \eqref{e23}, we have
\begin{eqnarray}\label{e25}
\E[{\gamma_n}]=
\frac{n+1-H_{n+1}}{2}
-\frac{1}{2^n}
\sum_{k=1}^{n-1}
\frac{(-1)^{k-1}\left(n+2-\binom{n}{k}\right)}
{k\binom{n+1}{k}}
+\frac{(-1)^n}{2n(n+1)},
\end{eqnarray}
In the first term $H_n\sim \ln n$, the second term which is the finite sum multiplied by $-2^{-n}$ is
$O(n/2^{n})$, and the last term is $O(1/n^{2})$. Hence, $\gamma_{avg}(W_n)\sim n/2$ follows.

Substituting \eqref{e23} and \eqref{e24} into the variance formulae in \eqref{e21} and \eqref{e22}, and noting that the finite sum in \eqref{e24} is exponentially small up to polynomial factors, gives
\[
\Var(\gamma_n)
=
\frac14\left(H_{n+1}-H_{n+1}^{(2)}\right)
+O\!\left(\frac{\log n}{n^2}\right).
\]
Using the asymptotics
\[
H_{n+1}\sim \ln n,\qquad
H_{n+1}^{(2)}\sim \frac{\pi^2}{6},
\]
we obtain $\Var(\gamma_n)\sim (\ln n)/4$. The proof is complete.
\end{proof}

\begin{proof}[\bf{Proof of Theorem \ref{thm:normal}}]
Let \(P_n(x)=\Gamma_{W_n}(x)/\Gamma_{W_n}(1)\) be the probability generating function of \(\gamma_n\). By Theorem 1.5, \(\Gamma_{W_n}\)(x) has only negative real roots; hence \(P_n(x)\) is real-rooted with all roots non-positive. By Theorem~\ref{thm:average},
\(
\mathrm{Var}(\gamma_n)\sim \frac14\ln n\to\infty.
\)
In Theorem 2 of \cite{Ben73}, Bender asserts that if a sequence of generating polynomials is real-rooted with all roots non-positive and has variance tending to infinity, then the corresponding normalized random variables converge in distribution to \(\mathcal{N}(0,1)\). Thus, \[
\frac{\gamma_n-\mathbb{E}[\gamma_n]}{\sqrt{\mathrm{Var}(\gamma_n)}}\xrightarrow{d}\mathcal{N}(0,1), \qquad\text{as } n\to\infty.
\]
Applying Theorem~\ref{thm:average} yields the asymptotics \(\mathbb{E}[\gamma_n]\sim\frac{n}{2}\) and \(\mathrm{Var}(\gamma_n)\sim\frac14\ln n\), the theorem follows. \end{proof}

\end{document}